\documentclass[reqno]{amsart}
\usepackage{amsmath,amssymb,cite}
\newtheorem{theorem}{Theorem}[section]
\newtheorem{proposition}[theorem]{Proposition}
\newtheorem{lemma}[theorem]{Lemma}
\numberwithin{equation}{section}
\numberwithin{theorem}{section}
\newcommand{\mc}[1]{{\mathcal #1}}
\newcommand{\bb}[1]{{\mathbb #1}}
\newcommand{\upbar}[1]{\,\overline{\! #1}}

\newcommand{\id}{{1 \mskip -5mu {\rm I}}}
\renewcommand{\epsilon}{\varepsilon}   
\title[GC functional for diffusion processes]{Small noise asymptotic
  of the Gallavotti-Cohen functional for diffusion processes}

\author[L.\ Bertini]{Lorenzo Bertini} 
\address{Lorenzo Bertini
  \hfill\break \indent
  Dipartimento di Matematica, Universit\`a di Roma `La Sapienza'
  \hfill\break \indent
  P.le Aldo Moro 2, 00185 Roma, Italy} 
\email{bertini@mat.uniroma1.it}

\author[G.\ Di Ges\`u]{Giacomo Di Ges\`u} 
\address{Giacomo Di Ges\`{u} 
  \hfill\break \indent
  Institut f\"{u}r Mathematik, 
  Technische Universit\"{a}t Berlin,
  \hfill\break \indent
  Stra\ss e des 17. Juni 136, 10623 Berlin, Germany} 
\email{digesu@math.tu-berlin.de}

\begin{document}

\noindent\keywords{Large deviations, Gallavotti-Cohen functional,
  Freidlin-Wentzell estimates} 

\subjclass[2000]{Primary 60F10, 82C05; Secondary 60J60}

\begin{abstract}
  We consider, for a diffusion process in $\bb R^n$, the
  Gallavotti-Cohen functional, defined as the empirical power
  dissipated in a time interval by the non-conservative part of the
  drift.  
  We prove a large deviation principle in the limit in which the noise
  vanishes and the time interval diverges.  The corresponding rate
  functional, which satisfies the fluctuation theorem, is expressed in
  terms of a variational problem on the classical Freidlin-Wentzell
  functional.
\end{abstract}

\maketitle
\thispagestyle{empty}

\section{Introduction}
\label{s:0}

The so-called Gallavotti-Cohen functional and the associated
fluctuation theorem have been originally introduced in the context of
chaotic deterministic dynamical systems \cite{GC}. Subsequently, they
have been extended to stochastic systems, originally in \cite{Kur} and
in more generality in \cite{LS}.  Since then, this topic has become a
basic issue in nonequilibrium statistical mechanics and it has even
been the object of true experiments \cite{Cil}. On the other hand,
the rigorous mathematical analysis has not been developed in great
detail.

We here consider the simple setting of a diffusion process in $\bb R^n$ 
with constant diffusion coefficient, defined as the solution of the
stochastic differential equation
\begin{equation}
  \label{se}	
  d\xi_t = c(\xi_t)dt + \sqrt{\epsilon} \, d\beta_t
\end{equation}
where $\beta$ is a standard $n$-dimensional Brownian motion, $c$ is a
smooth vector field on $\bb R^n$, and $\epsilon>0$ is the diffusion
coefficient. 
Under suitable assumptions on the drift $c$, there exists a unique
stationary distribution $\mu_\epsilon$ having a strictly positive
smooth density $\varrho_\epsilon:\bb R^n\to (0,+\infty)$ 
with respect to Lebesgue measure.
We recall that the corresponding stationary process, obtained by
choosing in $\eqref{se}$ the initial condition with law $\mu_\epsilon$, 
is reversible if and only if the drift is conservative, that is 
$c =- (1/2) \nabla V$ for some $V:\bb R^n\to \bb R$. 
In this case $\varrho_\epsilon(x) = C_\epsilon \exp\{-\epsilon^{-1}V(x)\}$.
When $c$ is not conservative, the forward and backward time evolutions
of the stationary process have different laws, and the
stationary distribution is not - in general - explicitly known.

The Gallavotti-Cohen functional is a random variable on the canonical
path space $C(\bb R_+;\bb R^n)$ expressed in terms of the
Radon-Nikodym derivative of the law of the stationary process with
respect to its time reversal.  The standard definition is the
following.  Let $\mc P^\epsilon_{\mu_\epsilon}$, a probability measure on
$C\big(\bb R_+;\bb R^n \big)$, be the law of the stationary
process.  By stationarity, $\mc P^\epsilon_{\mu_\epsilon}$ can be extended to a
probability measure on $C\big( \bb R; \bb R^n \big)$.  Denoting with
$\theta$ the time reversal, we set $\mc P^{\epsilon,*}_{\mu_\epsilon}= 
\mc P^\epsilon_{\mu_\epsilon} \circ \theta^{-1}$.  
Given $T>0$, the Gallavotti-Cohen functional can then be defined as
the random variable on $C(\bb R_+;\bb R^n)$ which is $\mc
P^\epsilon_{\mu_\epsilon}$ a.s.\ given by
\begin{equation}
\label{wrn}
\widetilde W_T = -\frac {\epsilon}{T} \log 
\frac {d\mc P^{\epsilon,*}_{\mu_\epsilon,T}}{d\mc P^\epsilon_{\mu_\epsilon,T}}        
\end{equation}
where the subscript $T$ denotes the measures induced by the
restriction to the time interval $[0,T]$ and we used a normalization in
which $\widetilde W_T$ is finite as $\epsilon \to 0$ and
$T\to \infty$.  
Of course, $\mc P^{\epsilon,*}_{\mu_\epsilon}= \mc
P^\epsilon_{\mu_\epsilon}$, i.e. $\widetilde W_T\equiv 0$, if
and only if the stationary process is reversible.
By denoting with $\mc E^\epsilon_{\mu_\epsilon}$ the expectation with
respect to $\mc P^\epsilon_{\mu_\epsilon}$, from Jensen inequality it
follows $\mc E^\epsilon_{\mu_\epsilon}\big( \widetilde
W_T \big)\ge0$.  In fact, $\mc
E^\epsilon_{\mu_\epsilon}\big(\widetilde W_T\big)$ is
proportional to the relative entropy of $\mc
P^{\epsilon,*}_{\mu_\epsilon,T}$ with respect to $\mc
P^{\epsilon}_{\mu_\epsilon,T}$, thus providing a natural measure of
irreversibility.

By denoting with $X_t$ the canonical coordinates on $C(\bb R_+;\bb
R^n)$, an informal computation \cite{LS} based on Girsanov formula
shows that $\mc P^{\epsilon}_{\mu_\epsilon}$ a.s.\
\begin{equation}
\label{wst}
\widetilde W_T 
= \frac {2}{T} \int_0^T \langle c(X_t), \circ dX_t \rangle 
- \frac{\epsilon}{T} \log \frac{\varrho_\epsilon(X_T)}{\varrho_\epsilon(X_0)}
\end{equation}
where $\langle\cdot,\cdot\rangle$ is the inner product in $\bb R^n$
and $\circ dX_t$ denotes the Stratonovich integral with respect to
$X_t$.
Once again, note that $\widetilde W_T$ vanishes if and only
if $c=-(1/2)\nabla V$. By writing the entropy balance, the random
variable $\widetilde W_T$ can finally be interpreted as the
empirical production of the Gibbs entropy \cite{LS}.

The content of the so-called \emph{fluctuation theorem} is the
following. Fix the diffusion coefficient $\epsilon>0$ and look for the
large deviations asymptotic of the family of probability measures on
$\bb R$ given by  
$\{\mc P^\epsilon_{\mu_\epsilon}\circ (\widetilde W_T)^{-1}\}_{T>0}$
as $T\to \infty$. Suppose they satisfy a large deviation principle
that we informally write as
\begin{equation}
  \label{ldtw}
  \mc P^\epsilon_{\mu_\epsilon} \big( W_T \approx q \big) 
  \asymp \exp\big\{ - T R^\epsilon (q) \big\}  
\end{equation}
then the odd part of the rate function $R^\epsilon$ is linear. More
precisely, with our choice of the normalization, $R^\epsilon(q)
-R^\epsilon(-q) = -\epsilon^{-1} q$.  This means that ratio between
the probability of the events $\{W_T \approx q \}$ and
$\{W_T \approx - q \}$ becomes fixed, independently of the
model, in the limit $T\to \infty$.  
Provided \eqref{ldtw} holds, a simple argument based on the definition
\eqref{wrn} and time reversal shows that the rate function
$R^\epsilon$ satisfies the fluctuation theorem.  Indeed, from the very
definition of $\mc P^{\epsilon, *}_{\mu_\epsilon}$, the statement of
the fluctuation theorem can be written as a true identity even for
finite $T$, see again \cite{LS}.

On the other hand, in the present setting of a diffusion process with
non compact state space, i.e.\ $\bb R^n$, it is not really so clear
that \eqref{ldtw} holds. The main issue is the possible unboundedness
of $c$ and the necessary unboundedness of $\log \varrho_\epsilon$ in
the decomposition \eqref{wst}.  In the case of a compact state space,
the standard procedure is the following \cite{LS}. Modify the
definition of the Gallavotti-Cohen functional by dropping from
\eqref{wst} the boundary term
$\log\varrho_\epsilon(X_T)/\varrho_\epsilon(X_0)$, which should not
matter in the limit $T\to \infty$. In other words, let $\widehat
W_T$ be the random variable on $C(\bb R_+;\bb R^n)$ which is
$\mc P^\epsilon_{\mu_\epsilon}$ a.s.\ defined by
\begin{equation}
  \label{wsthat}
  \widehat W_T 
  = \frac {2}{T} \int_0^T \langle c(X_t), \circ dX_t \rangle \;,
\end{equation}
so that $\widehat W_T$ can be interpreted as the empirical power
dissipated by the vector field $c$ in the time interval $[0,T]$.

Using Girsanov theorem, Feynman-Kac formula, and Perron-Frobenius
theorem, see \cite{LS}, we informally deduce that for each
$\lambda\in\bb R$
\begin{equation*}
  \lim_{T\to\infty} \frac {1}{T} 
  \log \mc E^\epsilon_{\mu_\epsilon} \Big( 
  \exp\big\{ \lambda T W_T \big\} \Big)
  = e^\epsilon (\lambda)
\end{equation*}
where $e^\epsilon (\lambda)$ is the eigenvalue with maximal
real part (which is real) of the differential operator
$A_{\epsilon,\lambda}$ on $\bb R^n$ given by
\begin{equation}
  \label{lel}
  \begin{split}
  A_{\epsilon,\lambda} f (x) =& 
  \frac{\epsilon}{2} \Delta f (x) 
  + (1+2\lambda\epsilon) \big\langle c(x),\nabla f (x) \big\rangle  
  \\
  &
  + 2 \lambda(1+\lambda\epsilon)  
  \big\langle c(x), c(x) \big\rangle \, f(x)
  + \lambda \epsilon \nabla\cdot c(x) \, f(x)
  \end{split}
\end{equation}
where $\Delta f$ is the Laplacian of $f$, $\nabla f$ is the gradient
of $f$, and $\nabla\cdot c$ is the divergence of $c$. From H\"{o}lder
inequality if follows that $e^\epsilon$ is a convex function; provided
it is also smooth, by using the G\"{a}rtner-Ellis theorem, we deduce
that the sequence of probability measures $\{\mc
P^\epsilon_{\mu_\epsilon} \circ(\widehat W_T)^{-1}\}_{T>0}$ satisfies
a large deviation principle as $T\to \infty$ with convex rate function
given by the Legendre transform of $e^\epsilon$.  The fluctuation
theorem then follows from the identity
$e^\epsilon(\lambda)=e^\epsilon(-\epsilon^{-1}-\lambda)$ which can be
easily checked.  Since the definition of $\widetilde W^\epsilon_T$ in
\eqref{wst} involves explicitly the density $\varrho_\epsilon$, which
in general is not known, the definition of $\widehat W^\epsilon_T$ in
\eqref{wsthat} is more concrete and somewhat more appealing.

In the present setting of a non compact space state, if the vector
field $c$ is unbounded, as it is typically the case, the maximal
eigenvalue $e^\epsilon(\lambda)$ becomes infinite.  We therefore need
to follow a different route.  In this paper, we assume the vector
field $c$ admits the decomposition $c= - (1/2)\nabla V + b$ which is
orthogonal in the sense that for each $x\in \bb R^n$ we have $\langle
\nabla V(x), b(x) \rangle =0$. We also assume that the potential
$V(x)$ is super-linear as $|x|\to \infty$ and that $b$ is a
nonconservative smooth bounded vector field with bounded derivatives.
In this setting, we \emph{define} the Gallavotti-Cohen functional as
the random variable on $C(\bb R_+;\bb R^n)$ which is
$\mc P^\epsilon_{\mu_\epsilon}$ a.s.\ defined by
\begin{equation}
  \label{1:w}
  W_T 
  = \frac {2}{T} \int_0^T \langle b(X_t), \circ dX_t \rangle 
\end{equation}
namely, as the empirical power dissipated by the nonconservative part
of the vector field $c$. If the (deterministic) dynamical system $\dot
x = c(x)$ admits a unique equilibrium solution $O$ which is globally
attractive, by classical Freidlin-Wentzell results
\cite[Thm.~4.3.1]{FW}, we have $\varrho_\epsilon(x) \asymp \exp\{-
\epsilon^{-1} [V(x)-V(O)]\}$.  Therefore, definition \eqref{1:w} is close to
\eqref{wst} with the obvious advantage that we cancelled the
unboundedness of $\log\varrho_\epsilon$ with the conservative part of
the drift $c$.  In view of the assumptions on the nonconservative part
$b$ of the drift, as here shown, it is possible to carry out the
analysis outlined below \eqref{wsthat} and prove a large deviation
principle for the sequence $\{\mc P^\epsilon_{\mu_\epsilon}\circ
(W_T)^{-1} \}_{T>0}$. The corresponding rate function $R^\epsilon$,
which is strictly convex, satisfies the fluctuation theorem.

\bigskip 
Having clarified the definition of the Gallavotti-Cohen functional in
a non compact state space, we next discuss the main topic of the
present paper, which is the analysis of the large deviation properties
of the sequence $\{\mc P^\epsilon_{\mu_\epsilon}\circ (W_T)^{-1}\}$ in
the joint limit $\epsilon\to 0$ and $T\to \infty$. The motivation for
such analysis is the following. While the fluctuation theorem
establish a general, \emph{model independent}, symmetry property of
the rate function, the rate function itself, which appears to be
experimentally accessible, may encode other, \emph{model dependent},
relevant properties of the system which may be best revealed by the
small noise limit.  For instance, as here argued, even if for fixed
$\epsilon>0$ the rate function $R^\epsilon$ is strictly convex, it may
happen that phase transitions occur in the small noise limit.

As suggested in \cite{Kur2}, one possibility is to consider the limit
$\epsilon\to 0$ \emph{after} the limit $T\to \infty$ and discuss the
asymptotic of the rate function $R^\epsilon$. Since the latter is
obtained as the Legendre transform of the maximal eigenvalue of a
differential operator similar to $A_{\epsilon,\lambda}$ in
\eqref{lel}, this immediately becomes a problem in semiclassical
limits. We here instead look at the asymptotic in the opposite order
of the limits namely, we consider \emph{first} the small noise limit
$\epsilon\to 0$ and \emph{then} the long time limit $T\to \infty$. 
If the vector field $c$ has a unique globally attractive attractor, we
expect that the order in which these limits are taken does not
matter. However, as here explicitly shown, this is not the case 
if there are several attractors.
From a physical viewpoint, the relevant order of the limits depends on
the details of the experimental setting; the analysis here performed
is applicable as long as the noise is small and the power dissipated
by the drift is measured over a time scale which is much shorter than
the (possible) metastable time scales.  From a mathematical viewpoint,
the asymptotic here considered amounts to analyze the variational
convergence, as $T\to\infty$, of the sequence of rate functions
describing the large deviations asymptotic of the sequence $\{\mc
P^\epsilon_{\mu_\epsilon}\circ (W_T)^{-1} \}_{\epsilon>0}$ for a fixed
$T>0$.  We mention that this order of the asymptotics is analogous to
the one discussed in the context of the fluctuations of the empirical
current in stochastic lattice gases \cite{BDGJL8,Der}, where the small
noise limit corresponds to the limit of infinitely many particles.
 
Our results are informally stated as follows.  We show that the
sequence of probability measures $\{\mc P^\epsilon_{\mu_\epsilon} \circ
(W_T)^{-1}\}_{\epsilon>0,T>0}$ satisfies the large deviation principle
\begin{equation*}
  \mc P^\epsilon_{\mu_\epsilon} \big( W_T \approx q \big) 
  \asymp \exp\big\{ - T \epsilon^{-1} s (q) \big\}  
\end{equation*}
as we let first $\epsilon\to 0$ and then $T\to\infty$. The associated
convex rate function $s:\bb R\to\bb R_+$ is given by the following
variational problem on the classical Freidlin-Wentzell rate functional
\begin{equation*}
  s(q) = \inf_{T>0} 
  \inf \Big\{ \frac 1{2T} \int_0^T\!dt \, 
   \big|\dot X_t - c(X_t)\big|^2 \: ,\; X\,:\:
   X_0=X_T\,,\:
   \frac 2T \int_0^T\!dt \, \langle b(X_t), \dot X_t\rangle = q \Big\}
\end{equation*} 
and satisfies the fluctuation theorem $s(q)-s(-q)=-q$.  In words,
$s(q)$ is obtained by minimizing the averaged Freidlin-Wentzell
functional over all closed paths for which the power dissipated by the
vector field $b$ equals $q$.  We remark that, differently from what
happens in the classical problem of the exit from a domain, the cost
is measured here per unit of time and the constraint depends on the
whole path.  By constructing a simple example, we also show that the
function $s$ is not - in general - strictly convex or continuously
differentiable.

\section{Notation and results}
\label{s:1}

We denote by $\langle \cdot,\cdot\rangle$ the canonical inner product
in $\bb R^n$ and by $|\cdot|$ the associated Euclidean norm. 
The gradient and the divergence in $\bb R^n$ are respectively denoted
by $\nabla$  and $\nabla\cdot$.
Fix a function $V:\bb R^n \to \bb R$ and a vector
field $b:\bb R^n\to \bb R^n$. We assume they 
satisfy the following conditions:
\begin{itemize}
\item[(\textbf{A})]{$V\in C^2(\bb R^n)\quad$ and 
$\qquad \displaystyle{
 \lim_{|x|\rightarrow\infty}
\frac{\big\langle \nabla V(x), x\big\rangle}{|x|} = +\infty\;;
}$
} 
\item[(\textbf{B})]{the vector field $b$ is not conservative and  
    $b\in C^1_\mathrm{b}(\bb R^n;\bb R^n)$, namely $b$ is a continuously
    differentiable bounded vector field with bounded derivatives;}
\item[(\textbf{C})]{for each $x\in\bb R^n$  
$\qquad \displaystyle{
\langle \nabla V(x), b(x) \rangle = 0\;.
}$
}
\end{itemize}

Observe that the assumption (\textbf{A}) implies a
superlinear growth of the potential $V$, i.e.\
\begin{equation*}
  \lim_{|x|\rightarrow\infty}\frac{V(x)}{|x|} = +\infty\;;
\end{equation*}
in particular, $\int\!dx\, \exp\{-\epsilon^{-1}V(x)\} <\infty$. 
We shall denote by $c:\bb R^n\to\bb R^n$ the continuously
differentiable vector field defined by
\begin{equation}
  \label{c=}
  c(x) := - \frac 12 \nabla V(x) + b(x) \;.
\end{equation}

Fix a standard filtered probability space $(\Omega,\mc F, \mc F_t,\bb P)$ 
carrying an $n$-dimensional Brownian motion $\beta$.  The
expectation with respect to $\bb P$ is denoted by $\bb E$.  Given
$\epsilon>0$ and $x\in\bb R^n$ we consider the stochastic differential
equation
\begin{equation}
  \label{6.1}
  \begin{cases}
  d\xi^{x}_t = c(\xi^{x}_t) dt + \sqrt{\epsilon} \, d\beta_t \;,\\
  \xi^{x}_0 =x \;.
  \end{cases}
\end{equation}
Assumptions (\textbf{A}) and (\textbf{B}), together with definition
\eqref{c=}, yield
\begin{equation*}
\lim_{|x|\to \infty}\big\langle c(x), x\big\rangle = -\infty \;.
\end{equation*} 
By standard criteria, see e.g.\ Theorems III.4.1 and III.5.1 in
\cite{Has}, we deduce the existence and uniqueness of a unique
nonexploding strong solution to \eqref{6.1} as well as the existence
of a unique stationary probability measure $\mu_\epsilon$ for the
Markov family $\{\xi^x, x\in \bb R^n\}$.  We consider the canonical
path space $C(\bb R_+;\bb R^n)$ endowed with the topology of uniform
convergence in compacts and the associated Borel $\sigma$-algebra. The
canonical coordinate on $C(\bb R_+;\bb R^n)$ is denoted by $X_t$.  We
denote by $\mc P^\epsilon_x$, a probability measure on $C(\bb R_+;\bb
R^n)$, the law of the process $\{\xi^x_t, \, t\in\bb R_+\}$. 
Given a Borel probability measure $\nu$ on $\bb R^n$ we set $\mc
P^\epsilon_\nu :=\int\!d\nu(x) \mc P^\epsilon_x$ and denote by $\mc
E^\epsilon_\nu$ the corresponding expectation.  In particular, $\mc
P^\epsilon_{\mu_\epsilon}$ is the law of the stationary process
associated to \eqref{6.1}.

As discussed in the introduction, given $\epsilon,T>0$ and a Borel
probability measure $\nu$ on $\bb R^n$, we let the
\emph{Gallavotti-Cohen functional} $W_{T}$ be the real random variable
on $C(\bb R_+;\bb R^n)$ which is $\mc P^\epsilon_\nu$ a.s.\ defined by
\begin{equation}
  \label{6.2}
  W_{T} := \frac 2T \int_0^T 
  \!\big\langle b(X_t), \circ dX_t \big\rangle 
  = \frac 2T \int_0^T \langle b(X_t), dX_t\rangle
  + \frac{\epsilon}{T}\int_0^T \!\!dt \: \nabla\cdot b \,(X_t) 
  \;.
\end{equation}
Here, $\circ dX_t$ denotes the Stratonovich integral with respect to
the semimartingale $X_t$. In the second equality we rewrote $W_{T}$ in
terms of the Ito integral by using \eqref{6.1} and the orthogonality
condition (\textbf{C}).

The main purpose of this paper is to analyze the large deviations
behavior of the family of Borel probability measures on $\bb R$ given
by $\{\mc P^\epsilon_\nu \circ (W_{T})^{-1}\}_{\epsilon>0,T>0}$ in
the joint limit $\epsilon\to 0$ and $T\to \infty$. We first discuss
briefly the typical behavior.  Fix $\epsilon>0$ and the initial
distribution $\nu$.  By \eqref{6.2} and the ergodicity of the process
$\xi^x$, see e.g.\ \cite[Thm. IV.5.1]{Has}, in the limit $T\to\infty$
we have $\mc P^\epsilon_\nu$ a.s.
\begin{equation*}
  \lim_{T\to\infty} W_{T}  = \int\!d\mu_\epsilon (x) \,
    \big[ 2 \big| b(x)\big|^2 + \epsilon \, \nabla\cdot b (x) \big] 
    \;.
\end{equation*}
If the dynamical system $\dot x = c(x)$ admits a unique
equilibrium solution $O$ which is globally attractive, $\mu_\epsilon$
converges weakly to $\delta_O$ as $\epsilon\to 0$. Hence, as $b(O)=0$,
\begin{equation*}
  \lim_{\epsilon\to 0} \: \lim_{T\to\infty}   W_{T} = 0\;.
\end{equation*}
In the presence of metastable states, e.g.\ when $c$ has several
critical points (and no other attractors), the sequence
$\{\mu_\epsilon\}$ concentrates on the critical points of $c$
corresponding to the deepest minima of $V$ and the above statement
still holds. Observe that this limiting behavior is independent on the
initial distribution $\nu$.  On the other hand, if we denote by
$\upbar{x}_t$ the solution to $\dot{\upbar{x}} = c(\upbar{x})$ with
initial condition $\upbar{x}_0 =x$ and choose $\nu =\delta_x$, 
for $T>0$ fixed we have
\begin{equation*}
  \lim_{\epsilon\to 0} W_{T}  = 
  \frac 2T \int_0^T\!dt \: \big| b(\upbar{x}_t)\big|^2 \;,
\end{equation*}
where the limit is in probability with respect to $\mc P^\epsilon_x$.  
If the dynamical system $\dot x = c(x)$ admits a unique equilibrium
solution which is globally attractive, in the limit as $T\to \infty$
we obtain the same result as before. On the other hand, if $c$ has not
a unique attractor, there are initial conditions $x\in \bb R^n$ such
that the limiting behavior of $W_T$ is different.  Thus - in general
- the order in which the limits $\epsilon\to 0$ and $T\to \infty$ are
taken is relevant.

\medskip 
In view of definition \eqref{6.2} and our assumptions on $V$ and $b$,
for $\epsilon>0$ fixed the large deviation principle for the sequence
of probability measures $\{ \mc P^\epsilon_{\mu_\epsilon} \circ
(W_{T})^{-1}\}_{T>0}$ can be proven along the same lines outlined in
the introduction.

\begin{theorem}
  \label{t:fixeps}
  Fix $\epsilon>0$ and assume the function $V:\bb R^n\to \bb R$ also
  satisfies
  \begin{equation}
    \label{cV}
    \lim_{|x|\to\infty} 
    \Big[ \big| \nabla V (x) \big|^2 - 2 \Delta V (x) \Big] = +\infty   
    \;.
  \end{equation}
  For each $\lambda\in\bb R$ the limit 
  \begin{equation*}
    e^\epsilon(\lambda) := 
    \lim_{T\to\infty} \frac 1T \log \mc E^\epsilon_{\mu_\epsilon}
    \Big( \exp\big\{ \lambda T W_{T} \big\} \Big)
  \end{equation*}
  exists and defines a convex real analytic function $e^\epsilon:\bb R \to
  \bb R$.  Let $R^\epsilon$ be the Legendre transform of $e^\epsilon$,
  i.e.\ $R^\epsilon(q) := \sup_\lambda \big\{ \lambda q -
  e^\epsilon(\lambda) \big\}$.  Then the family of Borel probability
  measures on $\bb R$ given by $\{ \mc P^\epsilon_{\mu_\epsilon} \circ
  (W_T)^{-1} \}_{T>0}$ satisfies a large deviation principle with
  speed $T$ and essentially strictly convex rate function
  $R^\epsilon$. Namely, for each closed set $\mc C \subset \bb R$ and
  each open set $\mc O\subset \bb R$ we have
 \begin{equation*}
   \begin{split}
   & \varlimsup_{T\to\infty} \frac 1T \log 
  \mc P^\epsilon_{\mu_\epsilon} \big( W_{T} \in \mc C \big) 
  \le - \inf_{q\in \mc C} R^\epsilon(q) \;,
  \\
  & \varliminf_{T\to\infty} \frac 1T \log 
  \mc P^\epsilon_{\mu_\epsilon} \big( W_{T} \in \mc O \big) 
  \ge - \inf_{q\in \mc O} R^\epsilon(q) \;.
  \end{split}
 \end{equation*}
 Finally, $R^\epsilon$ satisfies the 
 fluctuation theorem $R^\epsilon(q)- R^\epsilon(-q) = -\epsilon^{-1} q$.
\end{theorem}

\noindent\emph{Remark.} 
The large deviation principle stated above still holds, with the same
rate function, if the stationary measure $\mu_\epsilon$ 
is replaced by a probability measure $\nu_\epsilon$ such that 
$\nu_\epsilon(dx)= f_\epsilon(x) dx$ and  $f_\epsilon:\bb R^n\to \bb R_+$ is a
probability density satisfying $\int\!dx \, f_\epsilon(x)^2
\exp\{\epsilon^{-1}V(x)\} <\infty$.

\medskip 
Our next result is the large deviations asymptotic for the family of
probabilities $\{\mc P^\epsilon_x\circ(W_{T})^{-1}\}_{\epsilon>0,T>0}$ 
as we let \emph{first} $\epsilon\to 0$ and \emph{then} $T\to\infty$.
As discussed in the introduction, if the limits are taken in this
order it is possible to express the rate function as the solution of a
suitable variational problem.  Before stating the result, we define
the relevant rate function.  Given $T>0$, let
\begin{equation*}
  H_T :=  \Big\{ \varphi\in AC\big([0,T];\bb R^n\big) \, :\; 
  \int_0^T\!dt \: \big| \dot{\varphi}_t \big|^2 < \infty \Big\}
\end{equation*}
where $AC\big([0,T];\bb R^n\big)$ denotes the set of absolutely
continuous paths from $[0,T]$ to $\bb R^n$. 
Furthermore, given $x\in\bb R^n$, we set 
\begin{equation*}
  H^x_T :=  \big\{ \varphi\in H_T \,:\: \varphi_0=x \big\}\;.
\end{equation*}
For $x\in \bb R^n$ and $T>0$ we let $I^x_T : C\big([0,T];\bb R^n\big)
\to [0,+\infty]$ be the Freidlin-Wentzell rate functional associated
to \eqref{6.1} namely, 
\begin{equation}
 \label{4.3}
 I_T^x(\varphi) := 
 \begin{cases}
  {\displaystyle 
    \frac 12 \int_0^T\!dt \: \big| \dot \varphi_t -c(\varphi_t) \big|^2  
  } & \textrm{ if $\varphi\in H^x_T$}  \;,\\
  +\infty  & \textrm{ otherwise}  .
 \end{cases} 
\end{equation}
For $T>0$ we let $\mc L_T :H_T \to \bb R$ be the
power dissipated by the nonconservative vector field $b$ namely,
\begin{equation}
  \label{pow}
  \mc L_T (\varphi) := \frac 2T \int_0^T \!dt \, \langle b(\varphi_t) ,
  \dot \varphi_t \rangle\;.
\end{equation}
For $x,y\in \bb R^n$ and $q\in\bb R$ we then introduce the following
subsets of $C\big([0,T];\bb R^n\big)$
\begin{equation}
\label{4.4}
\begin{split}
   \mc A^x_T(q) &:=  \big\{ \varphi \in H^x_{T} \,:\: 
   \mc L_T(\varphi) = q \big\} 
   \\
   \mc A^{xy}_T(q) &:=  \big\{ \varphi \in H^{x}_T   \,:\:
   \varphi_T=y\,,\: \mc L_T(\varphi) = q  \big\}  
\end{split}
\end{equation}
and set 
\begin{equation}
  \label{Sxy}
   S^{xy}_T(q) := 
   \inf \big\{ I_T^x(\varphi) \:,\: \varphi\in \mc A^{xy}_T(q) \big\} \;.
\end{equation}
Given $x\in\bb R^n$, we finally define the function $s^x: \bb R \to \bb R_+$ by 
\begin{equation}
  \label{4.5}
  s^x(q) := \inf_{T>0} \frac 1T S^{xx}_T(q)\;.
\end{equation}
We show below, see Theorem~\ref{t:4.1}, that the function $s^x$ is in
fact independent of $x$. In the sequel, we therefore denote it simply by $s$.

\begin{theorem}
\label{t:mr}
The family of Borel probability measures on $\bb R$ given by 
$\{\mc P^\epsilon_x \circ (W_T)^{-1} \}_{\epsilon>0,T>0}$ satisfies,
as we let first $\epsilon\to 0$ and then $T\to\infty$, a large
deviation principle, uniform with respect to $x$ in compact subsets of
$\bb R^n$, with speed $\epsilon^{-1} T$ and convex rate function $s$.
Namely, for each nonempty compact $K\subset \bb R^n$, each 
closed set $\mc C \subset \bb R$, and each open set
$\mc O \subset \bb R$ we have
  \begin{eqnarray*}
    && \varlimsup_{T\to\infty} \; \varlimsup_{\epsilon\to 0} 
    \; \sup_{x\in K} \:
    \frac {\epsilon}T \log 
    \mc P^\epsilon_x \big( W_{T} \in \mc C \big) 
    \le - \inf_{q\in \mc C} s (q) \;,
    \\
    && 
    \varliminf_{T\to\infty} \; \varliminf_{\epsilon\to 0} 
    \; \inf_{x\in K} \:
    \frac {\epsilon}T \log 
    \mc P^\epsilon_x \big( W_{T}  \in \mc O \big) 
    \ge - \inf_{q\in \mc O} s(q) \;.
  \end{eqnarray*}
 Finally, $s$ satisfies the fluctuation theorem $s(q)-s(-q) = -q$.
\end{theorem}

\noindent\emph{Remark.} 
The above theorem, together with the \emph{exponential tightness} of the
family $\{\mu_\epsilon\}_{\epsilon>0}$, implies the large
deviation principle, with the same rate function, for the sequence 
$\{\mc P^\epsilon_{\mu_\epsilon} \circ (W_T)^{-1}\}_{\epsilon>0,T>0}$.

\medskip
The rest of the paper is organized as follows.
The proof of Theorem~\ref{t:fixeps} is detailed in
Section~\ref{s:appa}.  
In Section~\ref{s:stb} we discuss the large deviation principle for
the family $\{ \mc P^\epsilon_x \circ (W_T)^{-1} \}$ when $T$ is fixed
and $\epsilon\to 0$. By analyzing, in Section~\ref{s:ltb}, the
variational convergence of the associated rate functions as $T\to
\infty$ we then complete the proof of Theorem~\ref{t:mr}. 
Finally, in Section~\ref{s:ea} we give a simple example of a vector
field $c$ for which $s$ is not strictly convex.

\section{Large deviations principle as $T\to\infty$} 
\label{s:appa}

As the diffusion coefficient $\epsilon$ is here kept fixed, to
simplify the notation we set $\epsilon=1$ and drop the dependence from
$\epsilon$ from the notation. We also assume the arbitrary constant in
the definition of $V$ has been chosen so that $\exp\{-V(x)\}$ is a
probability density in $\bb R^n$ and denote by $\wp$ the probability
given by $d\wp(x)= \exp\{-V(x)\}\,dx$.  Let $\mc H$ be the
complex Hilbert space $\mc H=L^2(\bb R^n;d\wp)$; we denote by
$(\cdot,\cdot)$ and $\|\cdot\|$ the inner product and the norm in $\mc
H$.  Hereafter, we assume that the function $V:\bb R^n\to \bb R$
satisfies \eqref{cV} without further mention.

Fix a Borel probability measure $\nu$ on $\bb R^n$ and consider the
moment generating function of the random variable $tW_t$ with respect
to the probability $\mc P_\nu$.  
Recalling \eqref{6.2} and the assumption (\textbf{B}), a
straightforward application of the Girsanov formula yields the
following representation. For each $\lambda\in\bb R$
\begin{equation}
  \label{Girsanov}
  \mc E_\nu \,\exp\big\{ \lambda t W_{t} \big\} 
  = \int \!d\nu(x) \, 
  \bb E \, \exp \Big\{ \int_0^t\!ds \,
  U_\lambda\big(\xi^{\lambda,x}_s\big)
  \Big\}
\end{equation}
where, for $x\in\bb R^n$, the process $\xi^{\lambda,x}$ is the solution
to the stochastic differential equation
\begin{equation}
  \label{Xl}
  \begin{cases}
    \displaystyle{ 
      d \xi^{\lambda,x}_t = \Big[ -\frac12 \nabla V \big(\xi^{\lambda,x}_t\big)
      +(1+2\lambda) \, b\big(\xi^{\lambda,x}_t\big) \Big]dt + d\beta_t} 
    \\
     \xi^{\lambda,x}_0 =x
  \end{cases}
\end{equation}
and $U_\lambda :\bb R^n\to \bb R$ is the function defined by
\begin{equation}
  \label{Ul}
  U_\lambda (x) := 2 \lambda(1+\lambda) |b(x)|^2  
  +\lambda \nabla \cdot b(x) \;.
\end{equation}
 
For $\lambda\in\bb C$, let $A_\lambda^\circ$ be operator on $\mc H$
densely defined on $C^\infty_0(\bb R^n)$, the set of smooth functions
with compact support, by
\begin{equation}
  \label{al}
  A_\lambda^\circ f (x) := 
  \frac 12 \Delta f (x)
  - \frac 12 \langle \nabla V (x) , \nabla f(x) \rangle
  +  (1+2\lambda) \langle b(x), \nabla f(x) \rangle  + U_\lambda(x) f(x)
  \;.
\end{equation}
We denote by $A_\lambda$ the closure of $A_\lambda^\circ$; its domain  
$\mc D_\lambda$ is given by
\begin{equation*}
  \mc D_\lambda= \big\{f\in\mc H \,: \:  
  \exists \{f_n\}\subset C^\infty_0(\bb R^n) \text{ s.t. }
  f_n\to f \text{ in } \mc H \text{ and } A_\lambda^\circ f_n \text{ converges
    in } \mc H  \big\}\;.
\end{equation*}
The adjoint of $A_\lambda$ in $\mc{H}$ is denoted by $A_\lambda'$.

\begin{lemma}
  \label{t:ags}
  The set $\mc D_\lambda$ is independent on $\lambda$ and coincides
  with the domain of $A_\lambda'$ for any $\lambda\in \bb C$.
  Moreover, $A_\lambda$ generates a compact strongly continuous
  semigroup $T_t^\lambda$ on $\mc H$.  
  Finally, if $\lambda\in\bb R$ then the semigroup $T_t^\lambda$ is
  positive and irreducible and for $f\in\mc H$ the representation
  \begin{equation}
    \label{fk}
    T_t^\lambda f \, (x) = \bb E 
    \Big( \exp\Big\{\int_0^t\!ds \, U_\lambda(\xi^{\lambda,x}_s) \Big\}
    \, f\big(\xi^{\lambda,x}_t\big) \Big)
  \end{equation}
  holds.
\end{lemma}

Before proving this result, we clarify the terminology used.  The
semigroup $T^\lambda_t$ is \emph{positive} (positivity preserving in
the terminology of \cite{RS}) if $f\ge 0$ implies $T^\lambda_t f\ge 0$
for any $t\ge 0$.  Given $f\in\mc H$, we write $f>0$ iff $f\ge 0$ and
$f\neq 0$. We write $f\gg 0$ iff $f(x)> 0$ $\wp$ a.e.  Then, see
\cite[Def.~C-III.3.1]{Nagel}, the positive semigroup $T^\lambda_t$ is
\emph{irreducible} (ergodic in the terminology of \cite{RS}) if, given
$f>0$ and $g>0$ in $\mc H$, there exists $t_0\ge 0$ such that
$\big(g,T^\lambda_{t_0}f\big)>0$.  An equivalent characterization of
irreducibility is the following: for sufficiently large $z$ in the
resolvent set of the generator $A_\lambda$ the resolvent
$R_z=(z-A_\lambda)^{-1}$ is positivity improving, i.e.\ $f>0$ implies
$R_z f\gg 0$.

\begin{proof}
  Fix $\lambda\in \bb C$. On $C_0^\infty(\bb R^n)$ define the
  operators
  \begin{equation*}
    \begin{split}
      &H_0^\circ  f (x)= \frac 12 \Delta f (x)- \frac 12 
      \langle \nabla V (x), \nabla f(x)\rangle
      \\
      &H_{1}^\circ f(x) = (1+2\lambda) \, \langle b(x) , \nabla f(x)\rangle
      + U_\lambda(x) f(x)
    \end{split}
  \end{equation*}
  so that $A_\lambda^\circ=H_0^\circ + H_{1}^\circ$.

  \smallskip
  \noindent\emph{Step 1.}
  The operator $H_0^\circ$ with domain $C_0^\infty(\bb R^n)$ is
  essentially self-adjoint in $\mc H$ and its closure has compact
  resolvent.

  To complete this step, it is enough to consider the well-known
  \emph{ground state transformation} and apply standard criteria for
  Schr\"odinger operators. Here the details.  Let $U : \mc H \to
  L^2(\bb R^n;dx)$ be the isometry defined by $Uf = \exp\{- V/2\} f$.
  By an explicit computation
  \begin{equation*}
    U H_0^\circ U^{-1} = - \Big[ -\frac 12 \Delta + \frac 18 |\nabla V|^2
    -\frac 14 \Delta V \Big]
  \end{equation*}
  Recalling the assumption \eqref{cV}, the step now follows from
  \cite[Thm~X.29]{RS} and \cite[Thm~XIII.67]{RS}.
  We denote by $(H_0,\mc D)$ the closure of $H_0^\circ$.  
  Observe that
  \begin{equation}
    \label{H_0 negative}
    (H_0 f,f)=-\frac{1}{2}\|\nabla f\|^2
 \end{equation}
 which implies that $H_0$ is bounded from above by zero.

  \smallskip
  \noindent\emph{Step 2.}
  The operator $H_{1}^\circ$ is $H_0^\circ$-bounded with
  $H_0^\circ$-bound equal to $0$, i.e.\ for each $\gamma>0$ there
  exists a constant $C_\gamma$, also depending on $\lambda$, such that
  for any $f\in C^\infty_0(\bb R^n)$
  \begin{equation}
    \label{relativebound}
    \| H_{1}^\circ f\| \le \gamma \|H_0^\circ f\| + C_\gamma \|f\|
  \end{equation}
  Observe that this implies that $H_{1}^\circ$ can be uniquely extended to 
$\mc D$. We denote this extension by $(H_1,\mc D)$. 
  Moreover the bound \eqref{relativebound} 
  holds, with $H_0^\circ$ and $H_1^\circ$ replaced by $H_0$ and $H_1$, 
  for any $f\in \mc{D}$.  
  
  To prove \eqref{relativebound}, let $g\in\mc H$, $f\in
  C^\infty_0(\bb R^n)$. By assumption (\textbf{B}) and Cauchy-Schwarz
  inequality, for any $\gamma>0$ we have
  \begin{equation*} 
    \begin{split}
      \big( g,H_{1}^\circ f\big) 
      &\le (1+2|\lambda|) \int\!d\wp(x) \,
      |b(x)| \ |g(x)| \ |\nabla f(x)| + C \|g\| \, \|f\|
      \\
      &\le (1+2|\lambda|) \Big[ \frac{\gamma}{2} \| \nabla f \|^2 +
      \frac {1}{2\gamma} \, \sup_x |b(x)|^2 \, \|g\|^2 \Big] + C \|g\|
      \|f\|
    \end{split}
  \end{equation*}
  for some constant $C$ depending on $\lambda$ and the vector field
  $b$.  Thanks to \eqref{H_0 negative}, by taking the supremum over
  $g$ with $\|g\|=1$ and redefining $\gamma$, we deduce
  \begin{equation*}
  \| H_{1}^\circ f \| 
   \le  \gamma \|f\| \|H_0^\circ f\| + C_\gamma ( 1+\|f\|)
 \end{equation*}
 for some constant $C_\gamma\in(0,+\infty)$.  Replacing above $f$ with
 $f/\|f\|$ the step follows.

 \smallskip
 \noindent\emph{Step 3.}
 The operator $(H_{1},\mc D)$ is $(H_0,\mc{D})$-compact.
 
 We need to show that for some $z$ (hence for all $z$ in the resolvent
 set of $H_0$, which by \eqref{H_0 negative} is not empty) $H_{1}
 R_{0,z}$ is compact, where $R_{0,z}$ denotes the resolvent of the
 operator $(H_0,\mc D)$.  Fix a bounded sequence $\{f_n\}\subset \mc
 H$.  By Step~1 there exists a subsequence, still denoted by $f_n$,
 such that $\{R_{0,z}f_n\}$ converges. By linearity, we can assume
 that $R_{0,z}f_n \to 0$; we shall show that $H_{1} R_{0,z} f_n \to
 0$ as well. Indeed, by Step~2 for any $\gamma >0$ we have
 \begin{equation*}
   \| H_{1} R_{0,z} f_n\| \le \gamma \| H_0 R_z(H_0) f_n\| 
   + C_\gamma \| R_z(H_0) f_n\| \;.
 \end{equation*}
 Since for each $z$ in the resolvent set of $H_0$ the operator $H_0
 R_z(H_0)$ is bounded, the result follows by taking first the limit
 $n\to \infty$ and then $\gamma \downarrow 0$.

 \smallskip
 \noindent\emph{Step 4.}
 For every $\lambda\in \bb C$, $\mc D_\lambda=\mc{D}$ and $A_\lambda$
 generates a compact semigroup on $\mc H$.

 The first statement follows from \cite[Thm.~IV.1.11]{Kato}. Moreover,
 since by \eqref{H_0 negative} $H_0$ generates a strongly continuous
 holomorphic semigroup on $\mc H$, the same is true for $A_\lambda$ by
 \cite[Thm.~X.54]{RS} and Step~2.  One can easily check that for $z$
 sufficiently large the resolvent $R_z$ of $A_\lambda$ satisfies
 \begin{equation*}
   R_z = R_{0,z} + R_z H_1 R_{0,z}
 \end{equation*}
 hence, by Steps 1 and 3, the operator $A_\lambda$ has compact
 resolvent. Since a holomorphic strongly continuous semigroup whose
 generator has compact resolvent is always compact
 \cite[Thm.~II.4.29]{EN} this concludes the step.

 \smallskip
 \noindent\emph{Step 5.}
 The operator $A_\lambda'$ has domain $\mc D$.

 Observe that trivially the domain of $A_\lambda'$ contains $\mc D$.
 Define on $C_0^\infty$ the formal adjoint $H_1^*$ of $H_1$ in $\mc
 H$.  In view of assumption (\textbf{C}), an elementary computation
 shows that
 \begin{equation*}
   H_1^* f (x) = 
   -  (1+2\upbar{\lambda}) \langle b(x), \nabla f(x)\rangle
   + \upbar{U_\lambda(x)} f(x)
 \end{equation*}
 where $\upbar z$ is the complex conjugate of $z$.  Notice
 that Step~2 holds also with $H_1^\circ$ replaced by $H_1^*$. 
 As in Step~3, we can then conclude that the adjoint $H_1'$ of $H_1$
 is $(H_0,\mc D)$-compact.  By \cite{Beals} it follows that
 $A_\lambda'=(H_0,\mc D) + H_1$, so the domain of $A_\lambda'$ can not
 be larger than $\mc D$.

 \smallskip
 \noindent\emph{Step 6.}
 If $\lambda\in\bb R$, then the semigroup $T^\lambda_t := \exp\{ t
 A_\lambda\}$ is positive, irreducible and satisfies \eqref{fk}.

 In view of the assumptions (\textbf{B}), the function $U_\lambda$
 defined in \eqref{Ul} is continuous and bounded, hence \eqref{fk} for
 $f\in C^\infty_0(\bb R^n)$ follows from the classical Feynman-Kac
 formula. By monotone approximation it holds for every $f\in\mc H$. In
 particular $T^\lambda_t$ is positive. To prove irreducibility we show
 that the resolvent $R_z$ of $A_\lambda$ is positivity improving for
 $z$ sufficiently large.  Given $f\in\mc H$, $f>0$ we can find a
 function $g\in C^\infty_0(\bb R^n)$ such that $0<g \leq f$. For
 sufficiently large $z$ in the resolvent of $A_\lambda$, we have that
 $h:=R_z g\geq 0$ since the semigroup is positive. Furthermore
 $A_\lambda h -z h= -g <0$.  By elliptic regularity $h\in
 C^2(\bb{R}^n)$, so $h\neq 0$ and for large enough $z$ the maximum
 principle \cite[Thm.~II.6]{Protter} gives $h\gg0$.  Again by
 positivity of the semigroup $R_z f\geq R_z g = h\gg0$, which gives
 the irreducibility.
\end{proof}

Given $\lambda\in\bb R$, let $e(\lambda):=\sup\{ \textrm{Re\,}z,\,
z\in\textrm{spec\,} A_\lambda\}$ be the \emph{spectral bound} of the
operator $A_\lambda$. We recall, see \cite{Nagel}, that $e(\lambda)$ is
\emph{strictly dominant} iff there exists $\delta>0$ such that
$\textrm{Re\,}z\le e(\lambda)-\delta$ for any $z\in \textrm{spec\,}
A_\lambda\setminus\{e(\lambda)\}$. We recall also that an eigenvalue
is \emph{simple} iff it is a pole of order $1$ of the resolvent.

\begin{lemma}
  \label{t:spettr}
  Fix $\lambda\in\bb R$. Then $e(\lambda)$ is strictly dominant and
  a simple eigenvalue of $A_\lambda$ and of $A_\lambda'$. Moreover, there
  exist corresponding eigenvectors $\Psi_\lambda$ and $\Psi_\lambda'$
  of $A_\lambda$ and $A_\lambda'$ such that
  $\Psi_\lambda,\Psi_\lambda'\gg0$ and
  $(\Psi_\lambda,\Psi_\lambda')=1$.  
  In particular, the projection given by
  $P_\lambda=(\Psi'_\lambda,\cdot)\Psi_\lambda$ is positivity
  improving and there exist constants $\delta,C>0$ such that for any
  $t\ge 0$ 
  \begin{equation*} 
    \big\| T_t^\lambda - e^{t e(\lambda)} \, P_\lambda  \big\|
    \leq C\,  e ^{t[e(\lambda)-\delta]}
  \end{equation*}
  where $\|\cdot\|$ denotes the operator norm.
\end{lemma}

\begin{proof}
  The lemma follows from standard spectral theory of positive
  semigroups, see \cite[Thm.~C-III.1.1, Prop.~C-III.3.5]{Nagel} and
  \cite[Thm.~V.3.7]{EN}.  Notice that, by compactness and
  irreducibility of the semigroup, one can exclude that the spectrum
  is empty. This follows essentially from a result of de Pagter, see
  \cite[Thm.~C-III.3.7]{Nagel}.
\end{proof}

\noindent\emph{Remark.} 
If $\lambda=0$ the semigroup $T^\lambda_t$ is Markovian, i.e.\
$T^0_t1=1$. In this case $e(0)=0$ and $\Psi_0=1$.  It is also simple
to check that the stationary measure $\mu$ of the solution to
\eqref{6.1} for $\epsilon=1$, which coincides with \eqref{Xl} for
$\lambda=0$, is given by $d\mu=Z^{-1} \, \Psi'_0 d\wp$, where
$Z=\int\!d\wp(x) \, \Psi'_0(x)$ and $\Psi'_0$ is a strictly positive
eigenvector of $A_0'$ corresponding to the eigenvalue $e(0)=0$.

\begin{lemma}
  \label{t:anal} 
  The map $\bb R\ni \lambda\mapsto e(\lambda)$ is real analytic. 
\end{lemma}

\begin{proof}
  In view of Lemma~\ref{t:ags}, it is straightforward to check that
  the family of linear operators $\{A_\lambda\}_{\lambda\in\bb C} $ is
  a \emph{holomorphic family of type (A)} in the sense of
  \cite[VII~\S~2]{Kato}.  Fix $\lambda_0\in\bb R$. By the Kato-Rellich
  theorem \cite[Thm.~XII.8]{RS} there is an analytic function
  $\lambda\mapsto a(\lambda)\in \bb C$ such that
  $a(\lambda_0)=e(\lambda_0)$ and $a(\lambda)$ is the only eigenvalue
  of $A_\lambda$ near $e(\lambda)$. From the definition of
  $e(\lambda)$ and the analyticity of $a(\lambda)$ it easily follows
  that for $\lambda\in \bb R$ we have $a(\lambda)=e(\lambda)$, which
  proves the lemma.
\end{proof}

\begin{proof}[Proof of Theorem~\ref{t:fixeps}]
  The representation \eqref{Girsanov} and Lemma~\ref{t:ags} yield
  \begin{equation}
    \mc E_\mu \, \exp\big\{ \lambda t W_{t} \big\} 
    = \int \!d\mu(x) \, T_t^\lambda 1 \,(x)
  \end{equation}
  where $1$ denotes the function in $\mc H$ identically equal to one. 
  Using Lemma \ref{t:spettr} and the remark following it we get
  \begin{equation}
    \lim_{t\to\infty} 
    \frac{1}{t}\log \mc E_\mu \exp\big\{ \lambda t W_{t} \big\} 
    =  \lim_{t\to\infty} 
    \frac{1}{t}\log \big[ \ Z^{-1} \, \big(\Psi'_0  ,   T_t^\lambda 1
    \big)\ \big] 
    = e(\lambda) \;.
  \end{equation}

  An application of H\"older inequality yields the convexity of $e$.
  Moreover, in view of Lemma~\ref{t:anal}, the function $e$ is real
  analytic.  The large deviation principle for the family $\{\mc
  P_\mu\circ (W_T)^{-1}\}_{T>0}$ now follows from the G\"artner-Ellis
  theorem, see e.g.\ \cite[Thm.~2.3.6]{DZ}. Finally, by the smoothness
  of $e$ and \cite[Thm.~26.3]{Rock}, $R$ is essentially strictly
  convex.

  Recall that $A_\lambda'$ denotes the adjoint in $\mc H$ of the
  operator $A_\lambda$. By using assumption (\textbf{C}), an
  elementary computation shows that for $\lambda\in\bb R$ 
  the restriction to $C^\infty_0(\bb R^n)$ of the 
  operator $A'_\lambda$ is given  by
  \begin{equation*}
    \begin{split}
      A_\lambda' f(x) & = \frac 12 \Delta f(x) - \frac 12 \langle\nabla
      V(x), \nabla f (x)\rangle - (1+2\lambda) \langle b(x), \nabla
      f(x)\rangle 
      \\
      & + \big[ 2 \lambda(1+\lambda) |b(x)|^2 - (1+\lambda)
      \nabla \cdot b (x) \big] f(x)\;.
    \end{split}
  \end{equation*}
  In particular, by Lemma \ref{t:ags}, $A'_\lambda = A_{-1-\lambda}$.
  Lemma \ref{t:spettr} then gives $e(\lambda)=e(-1-\lambda)$ and the
  fluctuation theorem $R(q)-R(-q)=-q$ follows by taking Legendre
  transforms.
\end{proof}

\section{Small noise asymptotic for a fixed time interval}
\label{s:stb}

We start by proving the so-called \emph{exponential tightness} of the
family of probability measures $\{\mc P^\epsilon_x \circ
(W_{T})^{-1}\}_{\epsilon>0,T>0}$ as $\epsilon\to0$ and $T\to\infty$.
We emphasize that this result holds independently of the order of the
limits.

\begin{lemma}
  \label{t:expt}
  We have
  \begin{equation*}
    \lim_{\ell\to\infty}\; \varlimsup_{\substack{\epsilon\to 0 \\ T\to \infty}}
      \; \sup_{x\in \bb R^n}  \frac{\epsilon}{T} \log  
      \mc P^\epsilon_x  \big(  | W_{T}| > \ell \big) =- \infty \;.
  \end{equation*}
\end{lemma}

\begin{proof}
  We denote by $M$ the martingale part in the representation of $W_T$
  given in \eqref{6.2}, i.e.\
  \begin{equation*}
    M_{t}:= 
    \int_0^t \!  \big\langle b(X_s),  dX_s  - c(X_s) ds \big\rangle     
    = \int_0^t \!  \big\langle b(X_s),  dX_s  - b(X_s) ds \big\rangle     
  \end{equation*}
  which is indeed a martingale with respect to the probability 
  $\mc P^\epsilon_x$. 
  In view of assumption (\textbf{B}), $|b|^2$ and $\nabla \cdot b$ are
  bounded real functions.  It is thus enough to show
  \begin{equation*}
    \lim_{\ell\to\infty}\varlimsup_{\substack{\epsilon\to 0 \\ T\to \infty}}
    \; \sup_{x\in \bb R^n}
     \frac{\epsilon}{T} \log  
      \mc P^\epsilon_x \big( | M_{T} | > \ell \, T   \big) =- \infty \;.
  \end{equation*}
  Again by the boundedness of $b$, the process $M$ is a continuous
  martingale whose quadratic variation admits the bound 
  $[M]_t \le \epsilon B t$ with $B:= \sup_x |b(x)|^2$.  By
  applying the so-called Bernstein exponential inequality for
  martingales, see e.g.\ \cite[IV.3.16]{RY}, we get
  \begin{equation*}
    \mc P^\epsilon_x \Big( | M_{T} | >  \ell \, T \big) 
      \le 2 \exp\Big\{ -\frac{\ell^2 \,T}{2\,\epsilon B} \Big\}
  \end{equation*}
  which concludes the proof.
\end{proof}

In the remaining part of this section we prove the large deviation
principle for the family 
$\{\mc P^\epsilon_x\circ(W_{T})^{-1}\}_{\epsilon>0}$ for a fixed $T>0$.

\begin{theorem}
  \label{t:fixT}
  Fix $x\in \bb R^n$ and $T>0$. Recall \eqref{4.4}, \eqref{Sxy} and 
  let $S^x_T : \bb R\to [0,+\infty]$ be defined by 
  \begin{equation}
    \label{Sx}
    S^x_T (q) := \inf_{y\in \bb R^n} \, S^{xy}_T(q) 
    = \inf 
    \big\{ I_T^x(\varphi) \:,\; \varphi\in \mc A^{x}_T(q) \big\} \;.
  \end{equation}
  Let  $\{x_\epsilon\}\subset \bb R^n$ be a sequence converging to
  $x$. Then the family of Borel probability measures on $\bb R$ given
  by $\{\mc P^\epsilon_{x_\epsilon} \circ (W_{T})^{-1}\}_{\epsilon>0}$
  satisfies a large deviation principle with speed $\epsilon^{-1}$ and
  rate function $S^x_T$, i.e.\ for each closed set $\mc C \subset \bb
  R$ and each open set $\mc O \subset \bb R$ we have
  \begin{equation*}
    \begin{split}
      & \varlimsup_{\epsilon\to 0} \; {\epsilon} 
      \log \mc P^\epsilon_{x_\epsilon} \big( W_{T} \in \mc C \big) 
      \le - \inf_{q\in \mc C} S^x_T (q)
      \;, \\
      & \varliminf_{\epsilon\to 0} \; {\epsilon} 
      \log \mc P^\epsilon_{x_\epsilon} \big( W_{T} \in \mc O \big)
      \ge - \inf_{q\in \mc O} S^x_T(q) \;.
    \end{split}
  \end{equation*}
\end{theorem}

Since the random variable $W_T$ is not a continuous function on the
path space $C(\bb R_+;\bb R^n)$, this result cannot be obtained
directly from the usual Freidlin-Wentzell estimates \cite{FW} by
contraction principle and some work is needed.  We proceed by adding
$W_{t}$ as a $n+1$-th coordinate to the underlying diffusion process
$\xi^{x}_t$. We then exploit the Freidlin-Wentzell theory for this
extended (degenerate) diffusion process and finally project on the
coordinate we are interested in.

\begin{proof}
  Consider the map from $C(\bb R_+;\bb R^n)$ to $C(\bb R_+;\bb
  R^{n+1})$ which is $\mc P^\epsilon_x$ a.s.\ defined by $X_t\mapsto
  (X_t, tW_t)$ and denote by $\mc Q^\epsilon_x$ the push forward of
  $\mc P^\epsilon_x$. The probability measure $\mc Q^\epsilon_x$ on
  $C(\bb R_+;\bb R^{n+1})$ can be realized as the distribution of the
  $\bb{R}^{n+1}$-valued diffusion process $\eta^{x}$ which is the
  unique strong solution of the stochastic differential equation
  \begin{equation*}
    \begin{cases}
      d\eta^{x}_t = \big[{c}_0 (\eta^{x}_t) +\epsilon
      \,{c}_1(\eta^{x}_t)\big] dt
      + \sqrt{\epsilon} \, \sigma(\eta^{x}_t) d\beta_t \;,\\
      \eta^{x}_0=(x,0)\;.
    \end{cases}
  \end{equation*}
  Denoting by $y=(x,z)$ the coordinates in $\bb R^{n+1}$ with $x\in\bb
  R^n$ and $z\in \bb R$, the vector fields $c_0, c_1:\bb
  R^{n+1}\rightarrow \bb R^{n+1} $ are given by
  \begin{equation*}
    c_0(y)=\left(
      \begin{array}{c}
        c(x)\\ 2|b(x)|^2
      \end{array}
    \right) 
    \qquad\qquad
    c_1(y)=\left(
      \begin{array}{c}
        0 \\ \nabla\cdot c(x)
      \end{array}
    \right)
  \end{equation*}
  while $\sigma$ is the map from $\bb R^{n+1}$ to the set of
  $(n+1)\times n$ matrices given by
  \begin{equation*}
    \sigma (y)=\left(
      \begin{array}{c}
        \id_{n\times n} \\ 2b(x)
      \end{array}
    \right) 
  \end{equation*}
  where $\id_{n\times n}$ stands for the identity operator on $\bb R^n$.  
  The corresponding diffusion coefficient is the $(n+1)\times(n+1)$
  matrix given by
  \begin{equation*}
    a(y):=\sigma(y) \sigma(y)^{\intercal}
    =\left(
      \begin{array}{ccc}
        \id_{n\times n}
        & & 2b(x) \\
        \\
        2b(x)^\intercal & & 4|b(x)|^2
      \end{array}
    \right)\;.
  \end{equation*}
  Observe that $a(y)$ is singular for any $y\in \bb{R}^{n+1}$.  More
  precisely, $\textrm{Ker\,} a(y)$ is the one dimensional subspace
  spanned by the eigenvector $v_{0}(y):=(2 b(x),-1)$ corresponding to
  the zero eigenvalue of $a(y)$.

  Let $\mc Q^\epsilon_{x,T}$ be the restriction of the probability
  $\mc Q^\epsilon_{x}$ to the time interval $[0,T]$.  In order to
  obtain the large deviation asymptotic of the family $\{\mc
  Q^\epsilon_{x_\epsilon,T}\}_{\epsilon>0}$ as $\epsilon\to 0$, we
  need an extension of the classical Freidlin-Wentzell results to the
  case in which the drift depends on $\epsilon$ and the diffusion
  matrix is degenerate. Such a generalization is proven in
  \cite[Thm.~III.2.13]{Azencott} and gives the following.  Given
  $y\in\bb{R}^{n+1}$, let $a_{\mc V}(y)$ be the restriction of the
  linear operator $a(y)$ to $\mc V(y):=\big( \textrm{Ker\,} a(y)
  \big)^{\perp}$ and denote by $a_{\mc V}^{-1}$ be its inverse.  Then,
  as $\epsilon\to 0$, the family of probability measures on
  $C([0,T];\bb R^{n+1})$ given by $\{\mc Q^\epsilon_{x_\epsilon,T}
  \}_{\epsilon>0}$ satisfies a large deviation principle with speed
  $\epsilon^{-1}$ and rate function $J^x_T:C([0,T], \bb{R}^{n+1})\to
  [0,+\infty]$ given by
  \begin{equation*}
    J^x_T(\psi) =
    \begin{cases}
      \displaystyle{ \frac 12 \int_0^T\!dt \: \big\langle \dot \psi_t
        -{c}_0 (\psi_t)\,,\, a_{\mc V}^{-1}(\psi_t)\,[\dot \psi_t
        -{c}_0 (\psi_t)] \big\rangle}
      & \textrm{if $\psi\in \widetilde H^x_T$} \\
      +\infty & \textrm{ otherwise} \\
    \end{cases}
  \end{equation*}
  where
  \begin{equation*}
    \begin{split}
      \widetilde H^x_T := & \Big\{\psi\in AC([0,T];\bb{R}^{n+1})\,:\:
      \int_0^T\!dt \: \big| \dot{\psi}_t \big|^2 < \infty \,,\:  \\
      & \quad \psi_0=(x,0)\,,\: \textrm{ and } \: \dot \psi_t -{c}_0
      (\psi_t)\in \mc V(\psi_t) \textrm{ for a.e. } t\in [0,T] \Big \}\;.
    \end{split}
  \end{equation*}

  Fix $y=(x,z)\in\bb R^n\times\bb R$ and observe that a vector $\zeta$
  in $\bb R^{n+1}$ belongs to $\mc V(y)$ if and only if $\zeta= \big(
  \xi, 2\langle b(x), \xi\rangle \big)$ for some vector $\xi$ in $\bb
  R^n$.  Therefore, recalling \eqref{pow},
  \begin{equation*}
    \widetilde H^x_T = \Big\{ (\varphi,\chi)\in C([0,T];\bb{R}^n \times \bb{R})
    \, : \:  \varphi\in H^x_T\,,\: \chi_t =t\mc L_t(\varphi)  \Big \}\;.
  \end{equation*}

  An elementary computation shows that if $\zeta=\big( \xi, 2\langle
  b(x), \xi\rangle \big)$ for some $\xi$ in $\bb R^n$ then
  $\langle\zeta, a_{\mc V}^{-1}(y) \zeta \rangle = |\xi|^2$.  Consider
  the projection $\pi:C([0,T];\bb{R}^n \times \bb{R}) \to \bb{R}$
  given by $\pi(\varphi,\chi)=\chi_T / T$.  The previous statement and
  the last displayed equation imply
  \begin{equation*}
    S^x_T (q) = \inf\big\{ J^x_T(\varphi,\chi) \,,\:
    \pi(\varphi,\chi)=q \big\}\;.
  \end{equation*}
  The proof is then concluded by contraction principle.
\end{proof}

\section{Long time asymptotic of the rate function}
\label{s:ltb}

In this section we conclude the proof of Theorem~\ref{t:mr} by showing
that the sequence of rate functions $\{S^x_T/T\}_{T>0}$ converges to
$s$ as $T\to\infty$.  As the relevant quantities involved in the large
deviations asymptotic are the minima of $S^x_T$, the appropriate
notion of convergence is the so-called $\Gamma$-convergence.  We
recall its definition and basic properties, see e.g.\ \cite{Br}.  Let
$\mc X$ be a complete separable metric space (simply $\bb R$ in our
application); a sequence of functions $F_n : \mc X \to [0,+\infty]$ is
said to \emph{$\Gamma$-converge} to $F: \mc X \to [0,+\infty]$ iff the
two following conditions are satisfied for each $x\in \mc X$. For any
sequence $x_n\to x$ we have $\varliminf_n F_n(x_n)\ge F(x)$
(\emph{$\Gamma$-liminf inequality}).  There exists a sequence $x_n\to
x$ such that $\varlimsup_n F_n(x_n)\le F(x)$ (\emph{$\Gamma$-limsup
inequality}).  It is easy to show that $\Gamma$-converges of
$\{F_n\}$, together with its \emph{equi-coercivity}, implies the
convergence of the minima of $F_n$.  It is also easy to check that the
$\Gamma$-convergence of a sequence $\{F_n\}$ implies a lower bound on
the infimum over a compact set and an upper bound of the infimum over
an open set, see e.g.\ \cite[Prop.~1.18]{Br}.  Hence, in view of the
exponential tightness in Lemma~\ref{t:expt} and Theorem~\ref{t:fixT},
the proof of Theorem~\ref{t:mr} is a consequence of the following
result.

\begin{theorem}
\label{t:4.1}
The following statements hold.
\begin{itemize}
\item[(i)]{The function $s^x$, as defined in \eqref{4.5}, is
    independent of $x$ and convex. Furthermore, for each $x\in \bb
    R^n$ and $q\in \bb R$
    \begin{equation}
      \label{s=limT}
      s(q) = \lim_{T\to\infty} \frac {S^{xx}_T(q)}{T} \;.
    \end{equation}
    Finally, $s$ satisfies the fluctuation theorem $s(q)-s(-q)=-q$.
}
\item[(ii)]{Let $\{x_T\}\subset \bb R^n$ be a bounded 
    sequence. For $T>0$ let $s_T:\bb R\to [0,+\infty)$ be the 
    function defined by $s_T(q) := T^{-1} \, S^{x_T}_T(q)$. 
    Then, as $T\to\infty$, the sequence of real functions 
    $\{s_T\}_{T>0}$ $\Gamma$-converges to $s$. 
}
\end{itemize}
\end{theorem}

Postponing the proof of the above statements, we first 
show that they imply Theorem~\ref{t:mr}.

\begin{proof}[Proof of Theorem~\ref{t:mr}]
  We start by showing the large deviations upper bound.  Fix
  $T,\epsilon>0$. In view of the representation used in
  Section~\ref{s:stb}, the map $\bb R^n\ni x\mapsto \mc P^\epsilon_{x}
  \circ (W_T)^{-1}$ is continuous with respect to the weak topology of
  probability measures on $\bb R$. Given a closed set $\mc C\subset
  \bb R$, the map $\bb R^n \ni x \mapsto \mc P^\epsilon_{x} \big(W_T
  \in \mc C\big)\in \bb R$ is therefore upper semicontinuous. Hence,
  given a nonempty compact $K\subset \bb R^n$, there exists a sequence
  $\{x_{T,\epsilon}\}\subset K$ such that
  \begin{equation*}
    \sup_{x\in K} \, \mc P^\epsilon_{x} \big(W_T \in \mc C\big) = 
    \mc P^\epsilon_{x_{T,\epsilon}} \big(W_T \in \mc C\big) \;.
  \end{equation*}
  By taking, if necessary, a subsequence we may assume that
  $\{x_{T,\epsilon}\}_{\epsilon>0}$ converges to some $x_T\in K$. The
  large deviations upper bound in Theorem~\ref{t:fixT} now yields
  \begin{equation*}
   \varlimsup_{\epsilon\to 0}  \; \sup_{x\in K} \:
    \frac {\epsilon}T \log 
    \mc P^\epsilon_x \big( W_{T} \in \mc C \big) 
    \le  - \inf_{q\in\mc C} \, \frac 1T S^{x_T}_T(q) \;. 
  \end{equation*}

  In view of the exponential tightness proven in Lemma~\ref{t:expt},
  we may assume that $\mc C$ is a compact subset of $\bb R$. Since
  $\{x_T\}\subset K$, by item (ii) in Theorem~\ref{t:4.1} and
  \cite[Prop.~1.18]{Br} we deduce
  \begin{equation*}
    \varliminf_{T\to\infty} \: 
    \inf_{q\in\mc C} \, \frac 1T S^{x_T}_T(q) 
    = \varliminf_{T\to\infty} \: 
    \inf_{q\in\mc C} \, s_T(q)  \ge  \inf_{q\in\mc C} \, s(q)\;,
  \end{equation*}
  which concludes the proof the large deviations upper bound.

  The proof of the large deviations lower bound is analogous and the
  other statements follow directly from item (i) in Theorem~\ref{t:4.1}.
\end{proof}

In order to prove Theorem~\ref{t:4.1}, we start by the following
topological lemma.

\begin{lemma}
  \label{t:4.1.5}
  Fix $q\in \bb R$ and $x,y\in \bb R^n$. Then:
  \begin{itemize}
  \item[(i)]{for each $T>0$ the set 
  $\mc A^{xy}_T(q)$, as defined in \eqref{4.4}, is not empty;} 
  \item[(ii)]{there exist reals $T_0,C \in (0,\infty)$, 
      depending on $q$ and $x,y$, such that for any $T\ge T_0$ 
      \begin{equation*}
        S^{xy}_{T} (q) \le C \, T \;.
      \end{equation*}
  } 
  \end{itemize}
\end{lemma}

\begin{proof}
  The idea of the proof is quite simple. Since the vector field $b$ is
  not conservative, there exists a closed path for which the power
  dissipated by $b$ does not vanish. We thus only need to go from $x$
  to such closed path, repeat it an appropriate number of times,
  and then go to $y$. 

  Fix $q\in\bb R$ and $x,y\in\bb R^n$.  For suitable constants
  $T_0,C>0$, given $T\ge T_0$ we shall exhibit a path $\varphi\in \mc
  A^{xy}_T(q)$ such that $I^x_T(\varphi) \le C \,T$. Since the work
  done by the vector field $b$ along the path $\varphi$, i.e.\
  $\int_0^T\!dt\, \langle b(\varphi_t) , \dot \varphi_t \rangle$, is
  invariant with respect to reparameterizations of $\varphi$, this will
  prove both the statements of the lemma.

  Since the vector field $b$ is not conservative, there exists a point
  in $\bb R^n$, say $z$, and a closed path $\tilde\xi\in H^z_1$ such
  that $\tilde\xi_0=\tilde\xi_1=z$ and $\mc L_1 (\tilde\xi) =
  \upbar{q} \neq 0$. By denoting with $\hat\xi$ the time reversal of
  $\tilde\xi$ we also have $\mc L_1 (\hat\xi)= -\upbar{q}$.  By the
  continuity of the functional $\mc L_1$ on $H^z_1$ and the fact that
  $\bb R^n$ is simply connected, for each $p \in [ -|\bar q|, |\bar
  q|]$ there exists a path $\xi\in H^z_1$ such that $\xi_0=\xi_1=z$
  and $\mc L_1 (\xi) = p$.  Given $\lambda>0$ we let ${\xi}^\lambda
  \in \mc A^{zz}_{\lambda^{-1}}(\lambda p)$ be defined by
  ${\xi}^\lambda_t:= \xi_{\lambda t}$ and extend it by periodicity to
  a function defined on $\bb R$. Given $u,v\in \bb R^n$, let
  $\zeta^{uv}_t:= u + t(v-u)$, $t\in [0,1]$ and $\ell_{uv} := \mc
  L_1(\zeta^{uv})$.  Given a positive integer $N$ and $T_1\ge 0$ we
  then define the path $\varphi$, going from $x$ to $y$ in the time
  interval $[0,2+T_1+N\lambda^{-1}]$, by
  \begin{equation}
    \label{toppath}
    \varphi_t := 
    \begin{cases}
      \zeta^{xz}_t & t\in [0,1) \\
      {\xi}^\lambda_{t-1} & t \in [1,1+N\lambda^{-1}) \\
       \zeta^{zy}_{t-(1+N\lambda^{-1})} & t\in [1+N\lambda^{-1},2+N\lambda^{-1})\\
       y &  t\in [2+N\lambda^{-1},2+T_1+N\lambda^{-1}]\;.
    \end{cases}
  \end{equation}
  Then, by construction, 
  \begin{equation*}
    \mc L_{2+T_1+N\lambda^{-1}} (\varphi) = 
    \frac {1}{2+T_1+N\lambda^{-1}} \big( \ell_{xz} +\ell_{zy} + N p \big)
  \end{equation*}
  and 
 \begin{equation*}
    I^x_{2+T_1+N\lambda^{-1}} (\varphi) =  I^x_1(\zeta^{xz}) 
    + N I^z_{\lambda^{-1}} (\xi^\lambda) + I^z_1(\zeta^{zy}) + 
    \frac {T_1}2 \, |c(y)|^2  \;.
 \end{equation*}

 We now choose $\lambda=\lambda(q)>0$ such that $|q| \lambda^{-1} \le
 (1/2)|\bar q|$, $T_1\in [0,\lambda^{-1})$, and let
 $T_0=T_0(x,y,q)>2+\lambda^{-1}$ be such that
 \begin{equation*}
   \Big| \frac{(2+T_1) q - \ell_{xz} - \ell_{zy}}{\lambda (T_0-2) -1}\Big|
   \le \frac 12 \, |\bar q|\;.
 \end{equation*}
 For $T\ge T_0$ we next choose $N = [\lambda(T-2)]$, where $[\cdot]$
 denotes the integer part, $T_1= (T-2) - \lambda^{-1} [\lambda(T-2)]$,  
 and 
 \begin{equation*}
   p = \lambda^{-1} q + \frac{(2+T_1) q - \ell_{xz} - \ell_{zy}}{N} \;.
 \end{equation*}
 Note that $T_1\in [0,\lambda^{-1})$ and $p \in [-|\bar q|, |\bar q|]$
 by the previous choices.  As it is simple to check, the path
 $\varphi$ above constructed then satisfies $\mc L_{T}(\varphi) =q$
 and the bound $I^x_T (\varphi) \le C T$ for some $C=C(x,y,q)$
 independent of $T\ge T_0$.
\end{proof}

\begin{proof} [Proof of Theorem \ref{t:4.1}, item \emph{(i)}]
  Fix $x\in \bb R^n$, $q\in \bb R$ and $T_1,T_2>0$.  From
  Lem\-ma~\ref{t:4.1.5} and the goodness of the rate function $I^x_T$ it
  follows there exist $\varphi^i \in \mc A^{xx}_{T_i}(q)$ such that
  $S^{xx}_{T_i}(q)= I^x_{T_i}(\varphi^i)$, $i=1,2$.  By considering
  the path $\varphi_t$, $t\in [0,T_1+T_2]$ given by
  \begin{equation*}
  \varphi_t:= \id_{[0,T_1)}(t)\, \varphi^1_t +  
  \id_{[T_1,T_1+T_2]}(t)\, \varphi^2_{t-T_1}
  \end{equation*}
  we deduce that the sequence $\{S^{xx}_{T}(q)\}_{T>0}$ is
  subadditive, i.e.\
  \begin{equation*}
    S^{xx}_{T_1+T_2}(q) \le  S^{xx}_{T_1}(q)+ S^{xx}_{T_2}(q)\;.
  \end{equation*}
  Recalling \eqref{4.5}, the subadditivity just proven implies
  \eqref{s=limT}. By using the existence of the limit and again
  Lemma~\ref{t:4.1.5}, it is now simple to show that $s^x$ does not
  depend on $x$ and that it is convex.

  To prove the fluctuation theorem, observe that in view of
  \eqref{pow}, \eqref{Sxy}, and the orthogonality condition
  (\textbf{C})
  \begin{equation*}
    \frac 1T\, S^{xx}_T(q) = \inf\Big\{  \frac 1{2T} \int_0^T\!dt \, 
    \Big[ \big| \dot \varphi_t\big|^2 + \big|c(\varphi_t)\big|^2\Big] 
    - \frac 12 q \,,\: \varphi\in \mc A^{xx}_T(q) \Big\} \;.
  \end{equation*}
  In particular, if $\varphi$ is a minimizer for the right hand side
  above, then the path $\varphi^*$ defined by $\varphi^*_t:=
  \varphi_{T-t}$ is a minimizer for the analogous problem with $q$
  replaced by $-q$. Hence
  \begin{equation*}
    \frac 1T\, S^{xx}_T(q) -\frac 1T\, S^{xx}_T(-q) = -q 
  \end{equation*}
  and the statement follows by taking the limit $T\to \infty$.
\end{proof}

\begin{proof} [Proof of Theorem \ref{t:4.1}, item \emph{(ii)}, $\Gamma$-limsup
  inequality]
  Fix a sequence $T_n\to \infty$ and $q\in\bb R$. We need to show
  there exists a sequence $q_n\to q$ such that 
  $\varlimsup_n s_{T_n}(q_n) \le s(q)$.
  We claim it is enough to choose the constant sequence
  $q_n=q$. Indeed, letting $x_n:=x_{T_n}$, 
  from the very definition \eqref{Sx} of $S^{x}_T$ it follows  
  \begin{equation}
    \label{Glimsup}
    \varlimsup_n s_{T_n}(q) =
    \varlimsup_n \frac{1}{T_n} S^{x_n}_{T_n}(q) 
    \le  \varlimsup_n \frac{1}{T_n} S^{x_n x_n}_{T_n}(q) \;.
  \end{equation}
  Since $\{x_n\}$ is a bounded sequence, from
  Lemma~\ref{t:4.1.5} and \eqref{s=limT} it easily follows that 
  \begin{equation}
    \label{ney}
    \lim_n \frac{1}{T_n} S^{x_n x_n}_{T_n}(q) =s(q)
  \end{equation}
  which concludes the proof. 
\end{proof}

In order to prove the $\Gamma$-liminf inequality in Theorem~\ref{t:4.1},
we need to show that in the inequality in \eqref{Glimsup} we did not
loose much. On the other hand, if we let $\varphi^{T,x}$ be a
minimizer for the variational problem on the right hand side of
\eqref{Sx}, there is no reason for $\varphi^{T,x}_T$ to be equal to $x$.
In the next proposition, we show that we can extend $\varphi^{T,x}$ to
a path $\psi$ defined on a the longer time interval $[0,T+\tau]$ 
in such way that $\psi_{T+\tau}=x$ and the loss in the inequality in 
\eqref{Glimsup} becomes negligible as $T\to\infty$.

\begin{proposition}
  \label{t:allungo}
  Fix $q\in \bb R$, a bounded sequence $\{x_n\}\subset \bb R^n$, and
  sequences $T_n\to \infty$, $q_n\to q$. Let $\varphi^n\in \mc
  A^{x_n}_{T_n}(q_n)$ be such that $S^{x_n}_{T_n}(q_n) =
  I^{x_n}_{T_n}(\varphi^n)$ and set $y_n:= \varphi^n_{T_n}$. There
  exist sequences $\tau_n\to \infty$ and $\gamma^n\in \mc A^{y_n
    x_n}_{\tau_n}(\widehat{q}_n)$, where ${\widehat q}_n:= q +(q-q_n)
  T_n / \tau_n$, such that
  \begin{equation}
    \label{stallungo}
    \lim_n \frac {\tau_n}{T_n} =0 
    \qquad\qquad
     \lim_n \frac 1{T_n} \, I^{y_n}_{\tau_n} (\gamma^n) =0 \;.
  \end{equation}
\end{proposition}

Assuming the above proposition, we conclude the proof of
Theorem~\ref{t:4.1}.

\begin{proof} [Proof of Theorem \ref{t:4.1}, item \emph{(ii)},
  $\Gamma$-liminf inequality]
  Fix $q\in\bb R$, a sequence $T_n\to \infty$, and a sequence $q_n\to q$.  
  We need to show that $\varliminf_n s_{T_n}(q_n) \ge s(q)$. We define
  $x_n:=x_{T_n}$ so that $s_{T_n}(q_n) = T_n^{-1} S^{x_n}_{T_n}(q_n)$.

  Let $\varphi^n\in \mc A^{x_n}_{T_n}(q_n)$, $\{y_n\}\subset \bb R^n$,
  $\tau_n\to\infty$, $\{{\widehat q}_n\} \subset \bb R$, and
  $\gamma^n\in \mc A^{y_n x_n}_{\tau_n}({\widehat q}_n)$ be as in the
  statement of Proposition~\ref{t:allungo}.  Define the path
  $\psi^n\in H^{x_n}_{T_n+\tau_n}$ by
  \begin{equation*}
    \psi^n_t:= \id_{[0,T_n)}(t)\, \varphi^n_t +  
    \id_{[T_n,T_n+\tau_n]}(t)\, \gamma^n_{t-T_n} \;.
  \end{equation*}
  From the definition of ${\widehat q}_n$ it follows that 
  $\psi^n\in \mc A^{x_n x_n}_{T_n+\tau_n} (q)$. Since 
  \begin{equation*}
  I^{x_n}_{T_n+\tau_n} (\psi^n) = I^{x_n}_{T_n}(\varphi^n) 
  + I^{y_n}_{\tau_n} (\gamma^n) \;,
  \end{equation*}
  we have 
  \begin{equation*}
    S^{x_n}_{T_n} (q_n) 
    = I^{x_n}_{T_n+\tau_n} (\psi^n) -  I^{y_n}_{\tau_n} (\gamma^n) 
    \ge S^{x_n x_n}_{T_n+\tau_n}(q) -  I^{y_n}_{\tau_n} (\gamma^n) \;.
  \end{equation*}
  Divide the previous inequality by $T_n$ and take the liminf as $n\to
  \infty$. Since $\{x_n\}$ is bounded, the proof is achieved by using
  \eqref{ney} and \eqref{stallungo}.
\end{proof}

The two following lemmata are used in the proof of
Proposition~\ref{t:allungo}.  In the first one we show that the
endpoint of the minimizer for the variational problem on the right
hand side of \eqref{Sx} is at most at distance $O(T)$ from the initial
point $x$. In the second one we then show we can get back to a compact
independent of $T$ in a time which is $o(T)$.

\begin{lemma}
  \label{t:yn}
  Let  $\{y_n\} \subset \bb R^n$ be as in the statement of
  Proposition~\ref{t:allungo}. Then there exists a constant $C>0$ 
  such that $|y_n| \le C (T_n+1)$ for any $n\ge 1$. 
\end{lemma}

\begin{proof}
  Let $\{x_n\}\subset \bb R^n$, $T_n\to\infty$, $q_n\to q$ and
  $\varphi^n\in A^{x_n}_{T_n}(q_n)$ be as in the statement of
  Proposition~\ref{t:allungo}.  Since $\{x_n\}$ is bounded, by
  Lemma~\ref{t:4.1.5} there exists a constant $C_1>0$ independent of
  $n$ such that for any $n$ large enough $I^{x_n}_{T_n} (\varphi^n)
  \le C_1 T_n$.  Hence, by expanding the square in the definition
  \eqref{4.3} of $I^{x_n}_T$,
  \begin{eqnarray*}
    C_1 &\ge & \frac 1{T_n} \int_0^{T_n}\!dt \, 
    \Big[ \frac 12 | { \dot{\varphi} }^n_t  |^2 
    - \big\langle b(\varphi^n_t), { \dot{\varphi} }^n_t \big\rangle 
    + \frac 12
    \big\langle \nabla V(\varphi^n_t), { \dot{\varphi} }^n_t \big\rangle 
    \Big]
    \\
    & = & 
    \frac 1{2 T_n} \int_0^{T_n}\!dt \, | { \dot{\varphi} }^n_t  |^2  
    - \frac {q_n}2 +  \frac {V(y_n) - V(x_n)}{2 T_n} \;.
  \end{eqnarray*}
  Since $\{x_n\}$ is bounded and $\lim_{|x|\to\infty} V(x)=+\infty$,
  from the above bound it follows there exists a constant $C_2>0$
  independent of $n$ such that $\int_0^{T_n}\!dt \, | { \dot{\varphi}
  }^n_t |^2 \le C_2 T_n$. This yields $|y_n-x_n| \le C_3 T_n$ for some
  $C_3>0$ and concludes the proof.
\end{proof}

\begin{lemma}
  \label{t:arrivo}
  Let $\upbar{x}^y_t$, $t\in\bb R_+$, be the solution to
  $\,\dot{\!\upbar{x}} = c(\upbar{x})$ with initial condition
  $\upbar{x}_0=y$. There exists a compact $K\subset \bb R^n$
  independent of $y$ such that the following statement holds.  Denote
  by $\sigma_K(y) := \inf\{t \ge 0 \,:\, \upbar{x}^y_t \in K \}$ the
  hitting time of $K$; then
  \begin{equation*}
  \lim_{|y|\to\infty} \frac{\sigma_K(y)}{|y|} = 0 \;.
  \end{equation*}
\end{lemma}

\begin{proof}
  By assumption (\textbf{A}) there exists a Lipschitz function $\ell
  :\bb R_+ \to \bb R$ such that $\ell(r)\to +\infty$ as $r\to \infty$
  and for any $x\in \bb R^n$ $\big\langle \nabla V(x), x\big\rangle
  \ge 2\, \ell (|x|) \, |x|$.  Set $B:= \sup_{x} |b(x)|$ and let
  $R_0\in\bb R_+$ be such that $\ell(r) \ge B +1$ for any $r\ge R_0$.
  We claim the lemma holds with $K$ given by the ball of radius $R_0$.
  Indeed, recalling \eqref{c=} and using a standard comparison
  argument, if $|y|> R_0$ we have that $|\upbar{x}^y_t| \le r_t$ where
  $r_t$ solves the Cauchy problem
\begin{equation*}
  \begin{cases}
    \dot r = -  \ell(r) + B \;,\\
    r_0  = |y| \;.
  \end{cases}
\end{equation*}
Hence, 
\begin{equation*}
  \lim_{|y|\to\infty}  \frac{\sigma_K(y)}{|y|} \le 
  \lim_{|y|\to\infty}  \frac 1{|y|} 
  \int_{R_0}^{|y|}\!dr \,  \frac 1{\ell(r) -B} = 0
\end{equation*}
since $\ell(r)\to \infty$ as $r\to \infty$. 
\end{proof}

\begin{proof}[Proof of Proposition~\ref{t:allungo}]
Let $K\subset \bb R^n$ be as in Lemma~\ref{t:arrivo} and denote by
$\sigma_n\ge 0$ the hitting time of $K$ for the path $\upbar{x}$ with
initial condition $y_n$. Observe that,
by Lemmata~\ref{t:yn} and \ref{t:arrivo}, $\lim_n \sigma_n/ T_n=0$.
Set $z_n:= \upbar{x}_{\sigma_n}\in K$ and 
let $q^{(1)}_n := 2 \sigma_n^{-1} \int_0^{\sigma_n}\!dt\, 
\big|b(\upbar{x}_t)\big|^2$ be the power dissipated by vector field $b$
along the path $\upbar{x}_t$, $t\in[0,\sigma_n]$. Note that
$q^{(1)}_n$ is bounded by the boundedness of $b$. 

Choose a sequence ${\tilde \sigma}_n\to \infty$ such that
${\tilde \sigma}_n /T_n \to 0$, $\sup_n \sigma_n/{\tilde \sigma}_n
<\infty$, and $\sup_n (q-q_n)T_n/(\sigma_n+{\tilde \sigma}_n) <\infty$.
Set $\tau_n:= \sigma_n + {\tilde\sigma}_n$ and $q^{(2)}_n:=
{\tilde\sigma}_n^{-1} \big( \tau_n {\widehat q}_n -\sigma_n q^{(1)}_n\big)$.
Observe that $q^{(2)}_n$ is bounded by the boundedness of
$q^{(1)}_n$ and the choice of ${\tilde \sigma}_n$.  Let $\psi^n_t$, $t\in
[0,{\tilde\sigma}_n]$ be the path constructed in Lemma~\ref{t:4.1.5},
see in particular \eqref{toppath}, with $x$ replaced by $z_n$, $y$
replaced by $x_n$, $q$ replaced by $q^{(2)}_n$, and $T$ replaced by
${\tilde\sigma}_n$.

Finally, define the path $\gamma^n$ going form $y_n$ to $x_n$ in the
time interval $[0,\tau_n]$ by $\gamma^n_t:= \id_{[0,\sigma_n)}(t)\,
\upbar{x}_t + \id_{[\sigma_n,\tau_n]}(t) \, \psi^n_{t-\sigma_n}$. 
Then, by construction, 
\begin{equation*}
  \mc L_{\tau_n} (\gamma^n) 
  = \frac 1{\tau_n} \big[ \sigma_n q^{(1)}_n +{\tilde \sigma}_n 
   \mc L_{{\tilde\sigma}_n} (\psi^n) \big]
\end{equation*}
and 
\begin{equation*}
  I^{y_n}_{\tau_n}(\gamma^n) = 0 + I^{z_n}_1(\zeta^{z_nz}) 
  + N I^z_{\lambda^{-1}}(\xi^\lambda) + I^{z}_1(\zeta^{zx_n})  
  + \frac {T_1}2 \, |c(x_n)|^2 \;.   
\end{equation*}
Since $\{z_n\}\subset K$ and $\{x_n\}$, $\{q^{(2)}_n\}$ are bounded, we
conclude the proof by choosing sequences $\{\lambda_n\}\subset \bb
R_+$, $\{N_n\}\subset \bb N$, and $\{p_n\} \subset [-|\bar q|, |\bar
q|]$ as in Lemma~\ref{t:4.1.5}. Note in particular that with such
choices $\mc L_{{\tilde\sigma}_n} (\psi^n) = q^{(2)}_n$, $\lambda_n$
is bounded, and $\lim_n N_n/T_n=0$.
\end{proof}

\section{An example with not strictly convex rate function}
\label{s:ea}

We here given a simple example a two-dimensional vector field $c$
such that the rate function $s$ is not strictly convex.  Let
$U:\bb R_+ \to \bb R_+$ be a smooth function with two local
minima at $0$ and $R_0>0$ and super-linear growth as $r\to\infty$.  Set
$V(x) := U(|x|)$. Let also  the smooth vector field $b :\bb R^2 \to
\bb R^2$ be defined by $b(x) = A(|x|) x^\perp= A(|x|) \, (x^2,-x^1)$ 
where $A:\bb R_+\to \bb R_+$ is a smooth function with compact support in
$(0,+\infty)$ such that $A(R_0)>0$. 
Set finally $c(x):= -(1/2)\nabla V (x) + b(x)$.  Recalling $s$
has been defined in \eqref{4.5}, we claim that in this case it is not
strictly convex.

Consider the dynamical system
\begin{equation*}
  \dot x = -\frac 12 \nabla V(x) + b(x)\;.
\end{equation*}
By the above choices, it has the equilibrium solution $x=0$ and the
periodic solution $\bar x (t) = R_0 \big(\cos(\omega t), \sin(\omega
t)\big)$ where $\omega= A(R_0)/R_0$.  Let 
\begin{equation*}
  \bar q := \frac {2 \omega}{2 \pi} \int_0^{2\pi/\omega}\!dt \:
  \big\langle b\big(\bar x (t)\big), \dot{\bar x} (t) \big\rangle 
\end{equation*}
be the power dissipated along the periodic solution $\bar x$. Note that
$\bar q >0$.
By choosing the test paths $\varphi=0$ and $\varphi=\bar x$ we
deduce $s(0)=s(\bar q) =0$. Therefore, by the positivity and convexity
of $s$, we have $s(q)=0$ for any $q\in[0,\bar q]$. Moreover, the
fluctuation theorem $s(q) -s(-q) = -q$ implies that $s(q) = -q$ for
$q\in[-\bar q, 0]$.

\bigskip
\subsection*{Acknowledgements}
It is a pleasure to thank D.~Benedetto for his collaboration at an
initial stage of this work and D.~Fiorenza for his help on the
topological Lemma~\ref{t:4.1.5}. G.\ Di Ges\`u thanks M.~Klein and
K.-T.~Sturm for interesting discussions. The warm
hospitality and financial support of the University of Rome ``La Sapienza''
and of the IRTG-SMCP at TU Berlin are also acknowleged by the second author.

\end{document}